\documentclass[11pt,oneside,leqno]{amsart}
\usepackage{amsxtra}
\usepackage{amsopn}
\usepackage{booktabs}
\usepackage{multirow}
\usepackage{array}
\usepackage{color}
\usepackage{amsmath,amsthm,amssymb}
\usepackage{amscd}
\usepackage{amsfonts}
\usepackage{latexsym}
\usepackage{verbatim}
\usepackage{pb-diagram}
\usepackage{color}

\usepackage[all]{xy}

\theoremstyle{plain}
\newtheorem{theorem}{Theorem}[section]
\newtheorem*{theorem*}{Theorem}
\newtheorem{definition}[theorem]{Definition}

\newtheorem*{prop*}{Proposition}
\newtheorem{cor}[theorem]{Corollary}
\newtheorem*{cor*}{Corollary}
\newtheorem{rem}[theorem]{Remark}
\newtheorem{ex}[theorem]{Example}
\newtheorem*{mt*}{Main Theorem}
\DeclareMathOperator{\Imm}{Im}
\DeclareMathOperator{\Ker}{Ker}

\newcommand{\GL}{\hbox{\rm GL}}
\newcommand{\rr}{\mathbb{R}}
\newcommand{\cc}{\mathbb{C}}
\newcommand{\del}{\partial}
\newcommand{\delbar}{\overline{\del}}

\newcommand\SU{\hbox{\rm SU}}

\begin{document}
\date{\today}
\title[On geometric Bott-Chern formality and deformations]
{On geometric Bott-Chern formality and deformations}
\author{Nicoletta Tardini}
\address{Dipartimento di Matematica\\
Universit\`{a} di Pisa \\
Largo Bruno Pontecorvo 5, 56127\\
Pisa, Italy}
\email{tardini@mail.dm.unipi.it}
\author{Adriano Tomassini}
\address{Dipartimento di Matematica e Informatica\\
Universit\`{a} di Parma \\
Viale Parco Area delle Scienze 53/A, 43124 \\
Parma, Italy}
\email{adriano.tomassini@unipr.it}\thanks{Partially supported by GNSAGA
of INdAM and by the Project PRIN ``Variet\`a reali e complesse: geometria, topologia e analisi armonica''}
\begin{abstract} A notion of geometric formality in the context of Bott-Chern and Aeppli cohomologies on a complex manifold 
is discussed. In particular, by using Aeppli-Bott-Chern-Massey triple products, it is proved that geometric Aeppli-Bott-Chern formality is not 
stable under small deformations of the complex structure. 
\end{abstract}

\maketitle
%\tableofcontents
\section*{Introduction}

On a complex manifold one can consider two different kinds of invariants:
the topological ones of the underline manifold and the complex ones.
Among the first ones a fundamental role is played by de Rham cohomology,
among the second ones we recall the Dolbeault, Bott-Chern and Aeppli
cohomologies; where, Bott-Chern and Aeppli cohomologies of
a complex manifold $X$ are, respectively, defined as
\[
H_{BC}^{\bullet,\bullet}(X):=
\frac{\hbox{\rm Ker}\del\cap
\hbox{\rm Ker}\overline{\partial}}{\hbox{\rm Im}\del\delbar},\qquad
H_{A}^{\bullet,\bullet}(X):=
\frac{\hbox{\rm Ker}\del\delbar}
{\hbox{\rm Im}\del+
\hbox{\rm Im}\delbar}.
\]
Since all the cohomologies just mentioned coincide on a compact K\"ahler
manifold, more precisely $\del\delbar$-lemma holds on $X$ (this is
in particular true on a K\"ahler manifold) if and only if
the maps induced by identity in the diagram\\
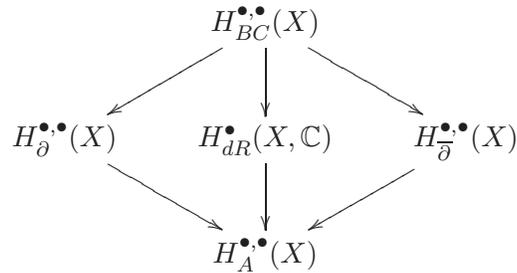
$$
\xymatrix{
 & H^{\bullet,\bullet}_{BC}(X) \ar[d]\ar[ld]\ar[rd] & \\
 H^{\bullet,\bullet}_{\del}(X) \ar[rd] & H^{\bullet}_{dR}(X,\cc) \ar[d] & H^{\bullet,\bullet}_{\delbar}(X) \ar[ld] \\
 & H^{\bullet,\bullet}_{A}(X) &
}.
$$
are all isomorphisms,
then Bott-Chern and Aeppli cohomologies could provide more informations
on the complex structure when $X$ does not admit a K\"ahler metric.\\
The theory of formality, developed by Sullivan, concerns with
differential-graded-algebras, namely graded algebras endowed with a derivation
with square equal to $0$. An immediate example is given by the space of differential
(resp. complex) forms on a differentiable (resp. complex) manifold together
with the exterior derivative. It is proven in \cite{DGMS} that compact complex manifolds
satisfying $\del\delbar$-lemma are formal in the sense of Sullivan.\\
On the other side, on a complex manifold $X$ the double complex
of bigraded forms
$\left(\Lambda^{\bullet,\bullet}X,\del,\delbar\right)$ is naturally defined;
then, one could ask whether a notion of formality could be defined
in case of bidifferential-bigraded-algebras. In this context Neisendorfer and Taylor developed a formality theory for the Dolbeault complex on complex manifolds (see \cite{NT}). In particular, we are interested
in a formality notion for Bott-Chern-cohomology.\\
Inspired by Kotschick \cite{K}, D. Angella and the second author in \cite{AT}, define a compact complex manifold $X$ being
geometrically-$H_{BC}$-formal if there exists a Hermitian metric $g$ on $X$
such that the space of $\Delta_{BC}$-harmonic forms (in the sense
of Schweitzer \cite{S}) has a structure of algebra. Moreover, an obstruction
to the existence of such a metric on $X$ is provided by {\em Aeppli-Bott-Chern-Massey
triple products} (see Theorem \ref{abc}).\\
In this note we are interested in studying the relationship
of this new notion with the complex structure, in particular
we discuss the behaviour of geometric-$H_{BC}$-formality
under small deformations of the complex structure (see \cite{TT} for similar results for Dolbeault formality).\\
Indeed, considering compact complex surfaces diffeomorphic to solvmanifolds
the property considered is open,
however, more in general, we prove the following

\smallskip
\noindent {\bfseries Theorem 1 (see Theorem \ref{instabilityCE} and Corollary \ref{instability}).\ }
{\itshape
 The property of geometric-$H_{BC}$-formality is not stable under small deformations
of the complex structure. 
}
\smallskip

A key tool in the proof of Theorem 1 is Theorem \ref{abc}. First of all we construct a complex curve ${J_t}$ of complex structures on  $X=\mathbb{S}^3\times\mathbb{S}^3$ such that 
$J_0$ is the the geometrically-$H_{BC}$-formal Calabi-Eckmann complex structure on $X$; then, by computing the Bott-Chern cohomology of $X_t=(\mathbb{S}^3\times\mathbb{S}^3,J_t)$ for small $t$, we exhibit a non-trivial Aeppli-Bott-Chern Massey triple product 
on $X_t$, for $t\neq 0$.
Furthermore, we show that the non holomorphically parallelizable Nakamura manifold has no geometrically $H_{BC}$-formal metric (see Example \ref{nakamura}).
\medskip

\noindent {\em Acknowledgements.} We would like to thank Daniele Angella for useful comments.
\smallskip

\section{Bott-Chern cohomology and Aeppli-Bott-Chern  geometrical formality}\label{preliminaries}
Let $X$ be a compact complex manifold of complex dimension $n$. We will denote by $A^{p,q}(X)$ the space of complex $(p,q)$-forms on $X$. The 
{\em Bott-Chern} and {\em Aeppli cohomology groups} of $X$ are defined respectively as (see \cite{A} and \cite{BC})
$$
H^{\bullet,\bullet}_{BC}(X)=\frac{\Ker\del\cap\Ker\delbar}{\Imm \del\delbar}\,,\quad
H^{\bullet,\bullet}_{A}(X)=\frac{\Ker\del\delbar}{\Imm \del+\Imm\delbar}\,.
$$
Let $g$ be a Hermitian metric on $X$ and $*:A^{p,q}(X)\to A^{n-p,n-q}(X)$ be the complex Hodge operator 
associated with $g$. Let $\tilde\Delta_{BC}$ and $ \tilde\Delta_{A}$ be the $4$-th order elliptic self-adjoint differential operators defined respectively as 
$$ 
\tilde\Delta_{BC}^g \;:=\;
\left(\del\delbar\right)\left(\del\delbar\right)^*+\left(\del\delbar\right)^*\left(\del\delbar\right)+\left(\delbar^*\del\right)\left(\delbar^*\del\right)^*+\left(\delbar^*\del\right)^*\left(\delbar^*\del\right)+\delbar^*\delbar+\del^*\del $$
and
$$ \tilde\Delta_{A}^g \;:=\; \del\del^*+\delbar\delbar^*+\left(\del\delbar\right)^*\left(\del\delbar\right)+\left(\del\delbar\right)\left(\del\delbar\right)^*+\left(\delbar\del^*\right)^*\left(\delbar\del^*\right)+\left(\delbar\del^*\right)\left(\delbar\del^*\right)^*\,.
$$
Then, accordingly to \cite{S}, it turns out that $H^{\bullet,\bullet}_{BC}(X)\simeq\Ker\tilde\Delta_{BC}^g$ and $H^{\bullet,\bullet}_{A}(X)\simeq\Ker\tilde\Delta_{A}^g$, so that 
$H^{\bullet,\bullet}_{BC}(X)$ and $H^{\bullet,\bullet}_{A}(X)$ are finite dimensional complex vector spaces. Denoting by $\alpha$ a $(p,q)$-form on $X$, note that
$$
\alpha\in\Ker\tilde\Delta_{BC}^g\quad\iff\quad
\left\{
\begin{array}{r}
\del\alpha=0\,,\\
\delbar\alpha=0\,,\\
\del\delbar*\alpha=0\,,
\end{array}
\right.
\quad \iff \quad*\alpha\in\Ker\tilde\Delta_{A}^g\,.
$$
Therefore,  $*$ induces an isomorphism between $H^{p,q}_{BC}(X)$ and $H^{n-p,n-q}_{A}(X)$. Furthermore, the wedge product induces a structure of algebra on 
$\bigoplus_{p,q} H^{p,q}_{BC}(X)$ and a structure of $\bigoplus_{p,q} H^{p,q}_{BC}(X)$-module on
$\bigoplus_{p,q} H^{p,q}_{A}(X)$ (see \cite[Lemme 2.5]{S}). 

Since in general the wedge product of harmonic forms may be not a harmonic form, the following definition makes sense (see \cite{K} for Riemannian metrics)
\begin{definition}
A Hermitian metric $g$ on $X$ is said to be {\em geometrically-$H_{BC}$-formal} if $\Ker\tilde\Delta_{BC}^g$ is an algebra. Similarly, a compact complex manifold $X$ is said to be {\em geometrically -$H_{BC}$-formal} if there exists a  geometrically $H_{BC}$-formal Hermitian metric on $X$.
\end{definition}

\section{Aeppli-Bott-Chern-Massey triple products}
Let $X$ be a compact complex manifold and denote by $(A^{\bullet,\bullet}(X),\del,\delbar)$ the bi-differential bi-graded algebra of $(p,q)$-forms on $X$.
As we have already noted in section \ref{preliminaries}, on a
compact complex manifold $X$, the Bott-Chern cohomology has
a structure of algebra, instead, the Aeppli cohomology has a structure
of $H^{\bullet,\bullet}_{BC}(X)$-module. This motivates the following (see  \cite{AT})
\begin{definition}
 %Let $\left( A^{\bullet,\bullet},\, \del,\, \delbar \right)$ be a bi-differential $\Z^2$-graded algebra. 
 Take
 $$
 \mathfrak{a}_{12} \;=\; \left[\alpha_{12}\right] \in H^{p,q}_{BC}(X) \;,
 \quad
 \mathfrak{a}_{23} \;=\; \left[\alpha_{23}\right] \in H^{r,s}_{BC}(X) \;,
 \quad
 \mathfrak{a}_{34} \;=\; \left[\alpha_{34}\right] \in H^{u,v}_{BC}(X) \;,
 $$
 such that $\mathfrak{a}_{12} \cup \mathfrak{a}_{23}=0$ in $H^{p+r,q+s}_{BC}(X)$ and $\mathfrak{a}_{23} \cup \mathfrak{a}_{34}=0$ in $H^{r+u,s+v}_{BC}(X)$: let
 $$
 (-1)^{p+q} \, \alpha_{12} \wedge \alpha_{23} \;=\; \del\delbar \alpha_{13}
 \qquad \text{ and } \qquad
 (-1)^{r+s} \, \alpha_{23} \wedge \alpha_{34} \;=\; \del\delbar \alpha_{24} \;.
 $$
 The {\em Aeppli-Bott-Chern triple Massey product} is defined as
\[
 \mathfrak{a}_{1234}
 \;:=\;
 \left\langle \mathfrak{a}_{12},\, \mathfrak{a}_{23},\, \mathfrak{a}_{34} \right\rangle_{ABC}
 :=
 \left[ (-1)^{p+q}\, \alpha_{12} \wedge \alpha_{24} - (-1)^{r+s}\, \alpha_{13} \wedge \alpha_{34} \right]\in
 \]
 \[\in \frac{H^{p+r+u-1, q+s+v-1}_{A}(X)}{H^{p,q}_{BC}(X) \cup H^{r+u-1, s+v-1}_{A}(X) + H^{p+r-1, q+s-1}_{A}(X) \cup H^{u,v}_{BC}(X)} \;.
\]
\end{definition}
% This link between Bott-Chern
%and Aeppli cohomologies deeply influences the previous definition.
%The definition does not depend on the representatives, see \cite{AT}.\\
Similarly to the real case, Aeppli-Bott-Chern Massey triple products provide an
obstruction to geometric-$H_{BC}$-formality (see also \cite{AT}).
\begin{theorem}\label{abc}
Let $X$ be a compact complex manifold. If $X$ is geometrically-$H_{BC}$-formal
then the Aeppli-Bott-Chern Massey triple products are trivial.
\end{theorem}
\begin{proof}
Fix a Hermitian metric $g$ on $X$ such that $\hbox{\rm Ker}\Delta_{BC}^{g}$ has a
structure of algebra. Take
$$
 \mathfrak{a}_{12} \;=\; \left[\alpha_{12}\right] \in H^{p,q}_{BC}(X) \;,
 \quad
 \mathfrak{a}_{23} \;=\; \left[\alpha_{23}\right] \in H^{r,s}_{BC}(X) \;,
 \quad
 \mathfrak{a}_{34} \;=\; \left[\alpha_{34}\right] \in H^{u,v}_{BC}(X) \;,
 $$
 such that $\mathfrak{a}_{12} \cup \mathfrak{a}_{23}=0$ in $H^{p+r,q+s}_{BC}(X)$
 and $\mathfrak{a}_{23} \cup \mathfrak{a}_{34}=0$ in $H^{r+u,s+v}_{BC}(X)$, with
 $\alpha_{12}, \alpha_{23},\alpha_{34}$ harmonic representatives in the respective
 classes. Then $\alpha_{12}\wedge\alpha_{23}$ and $\alpha_{23}\wedge\alpha_{34}$
 are harmonic forms with respect to the Laplacian $\tilde\Delta_{BC}^{g}$. Hence,
 with the notation introduced, $\alpha_{13}=0$ and $\alpha_{24}=0$. Therefore,
 by definition, $\left\langle \mathfrak{a}_{12},\, \mathfrak{a}_{23},\,
 \mathfrak{a}_{34} \right\rangle_{ABC}=0$.
\end{proof}
\begin{ex}\label{nakamura}
Let $G=\cc \ltimes_{\varphi} \cc^2$, where $\varphi:\cc\to\GL(2,\cc)$ is defined as 
$$
\varphi(x+iy)=
\left(
\begin{matrix}
e^x & 0\\
0 & e^{-x}
\end{matrix}
\right)\,.
$$
Then for some $a\in\rr$ the matrix $
\left(
\begin{matrix}
e^a & 0\\
0 & e^{-a}
\end{matrix}
\right)
$ is conjugated to a matrix in $\hbox{\rm SL}(2,\mathbb{Z})$. Then $\Gamma:=(a\mathbb{Z}+i2m\pi\mathbb{Z})\ltimes_\varphi \Gamma''$, 
with $\Gamma''$ lattice in $\cc^2$,  is a lattice in $G$ (see 
\cite{angella-kasuya} and \cite{Nak}). Denoting with $(z_1,z_2,z_3)$ global coordinates on $G$, the following forms
$$
\psi^1=dz_1\,,\qquad \psi^2=e^{-\frac{1}{2}(z_1+\overline{z}_1)}dz_2\,,\qquad \psi^3=e^{\frac{1}{2}(z_1+\overline{z}_1)}dz_3
$$  
are $\Gamma$-invariant. A direct computation shows that
$$
\begin{array}{lll}
\del \psi^1 =0 & \del\psi^2 = -\frac{1}{2}\psi^{12} & \del\psi^3 = \frac{1}{2}\psi^{13}\\[10pt]
\delbar \psi^1 =0 & \delbar\psi^2 = \frac{1}{2}\psi^{2\overline{1}} & \delbar\psi^3 = -\frac{1}{2}\psi^{3\overline{1}}\,,
\end{array}
$$
where $\psi^{A\overline{B}}=\psi^A\wedge\overline{\psi}^B$ and so on. Therefore $\{\psi^1,\psi^2,\psi^3\}$ give rise to complex $(1,0)$-forms on the compact manifold $X=\Gamma\backslash G$. 
We will show that $X$ is not geometrically-$H_{BC}$-formal. Let 
$$
\begin{array}{ll}
\mathfrak{a}_{12}=[e^{\frac12(z_1-\overline{z}_1)}\psi^{1\overline{3}}]\in H^{1,1}_{BC}(X)\,, &\mathfrak{a}_{23}=[e^{\frac12(z_1-\overline{z}_1)}\psi^{\overline{1}\overline{2}}]
\in H^{0,2}_{BC}(X)\,,\\[10pt]
\mathfrak{a}_{34}=[e^{-\frac12(z_1-\overline{z}_1)}\psi^{\overline{1}3}]\in H^{1,1}_{BC}(X)\,. &{}
\end{array}
$$ 
Then,
$$
e^{z_1-\overline{z}_1}\psi^{1\overline{1}\overline{2}\overline{3}}=\del\delbar (-e^{z_1-\overline{z}_1}\psi^{\overline{2}\overline{3}})\,.
$$
Therefore,
$$
\left\langle \mathfrak{a}_{12},\, \mathfrak{a}_{23},\, \mathfrak{a}_{34} \right\rangle_{ABC}
 =
 \left[e^{\frac12(z_1-\overline{z}_1)}\psi^{\overline{1}\overline{2}\overline{3}3}  \right]
 \in \frac{H^{1, 3}_{A}(X)}{H^{1,1}_{BC}(X) \cup H^{0,2}_{A}(X) } \;.
$$
According to the cohomology computations in \cite[table 4]{angella-kasuya}, it follows that 
$e^{\frac12(z_1-\overline{z}_1)}\psi^{\overline{1}\overline{2}\overline{3}3}\in \Ker\tilde\Delta_{A}^g$. Furthermore, a direct computation shows that 
$[e^{\frac12(z_1-\overline{z}_1)}\psi^{\overline{1}\overline{2}\overline{3}3}]\notin H^{1,1}_{BC}(X) \cup H^{0,2}_{A}(X)$, so that $\left\langle \mathfrak{a}_{12},\, \mathfrak{a}_{23},\, \mathfrak{a}_{34} \right\rangle_{ABC}$ is a non-trivial Aeppli-Bott-Chern Massey product. 
\end{ex}

\section{Instability of Bott-Chern  geometrical formality}
\begin{comment}
In the previous section, we saw that for compact complex surfaces diffeomorphic
to solvmanifolds and for class $VII$ surfaces with $b_2=0$,
namely Hopf surface and Inoue surfaces, the
geometrical-$H_{BC}$-formality is stable under small deformations of the complex structure. 
\end{comment}
In this section, starting with a geometrical-$H_{BC}$-formal compact
complex manifold, we will construct a complex deformation which is no more
geometrically-$H_{BC}$-formal. 

Let $X=\mathbb{S}^3\times\mathbb{S}^3$ and $\mathbb{S}^3\simeq\SU(2)$ be the Lie group of special unitary $2\times 2$ matrices and denote by $\mathfrak{su}(2)$ the Lie algebra of $\SU(2)$. Denote by 
$\{e_1, e_2, e_3\}$, $\{f_1, f_2, f_3\}$  a basis of the first copy of $\mathfrak{su}(2)$, respectively of the second copy of $\mathfrak{su}(2)$ and by  
$\{e^1, e^2, e^3\}$, $\{f^1, f^2, f^3\}$ the corresponding dual co-frames. Then we have the following commutation rules:
$$
[e_1,e_2]=2e_3\,,\quad [e_1,e_3]=-2e_2,\quad [e_2,e_3]=2e_1\,,
$$
and the corresponding Cartan structure equations
\begin{equation}
\left\{\begin{array}{rcl}
            de^1 &=&          -  2\, e^2 \wedge e^3 \\[5pt]
            de^2 &=& \phantom{+} 2\, e^1 \wedge e^3 \\[5pt]
            de^3 &=&          -  2\, e^1 \wedge e^2 \\[5pt]
            df^1 &=&          -  2\, f^2 \wedge f^3 \\[5pt]
            df^2 &=& \phantom{+} 2\, f^1 \wedge f^3 \\[5pt]
            df^3 &=&          -  2\, f^1 \wedge f^2
           \end{array}\right. \;.
\end{equation}
Define a complex structure $J$ on $X$ by setting 
$$
Je_1=e_2,\,\quad Jf_1=f_2,\,\quad Je_3=f_3\,.
$$
Note that $J$ is a Calabi-Eckmann structure or its conjugate on $\mathbb{S}^3\times\mathbb{S}^3$ (see \cite{C-E} and \cite{P}).

We have the following
\begin{theorem}\label{instabilityCE}
Let $X=\mathbb{S}^3\times\mathbb{S}^3$ be endowed with the complex structure $J$. Then $X$ is geometrically-$H_{BC}$-formal and there exists
a small deformation $\left\lbrace X_t\right\rbrace$ of $X$ such that $X_t$
is not geometrically-$H_{BC}$-formal for $t\neq 0$.
\end{theorem}
\begin{proof}
For the sake of the completeness we will recall the proof of geometric $H_{BC}$-formality of $X$ (see \cite{AT}).
According to the previous notation, a complex co-frame of $(1,0)$-forms for $J$ is given by 
\[
\left\{\begin{array}{rcl}
            \varphi^1 &:=& e^1 + i\, e^2 \\[5pt]
            \varphi^2 &:=& f^1 + i\, f^2 \\[5pt]
            \varphi^3 &:=& e^3 + i\, f^3
\end{array}\right. \;.
\]
Therefore the complex structure equations are given by
\[
\left\{\begin{array}{rcl}
            d\varphi^1 &=&     i\varphi^1\wedge\varphi^3+
            					i\varphi^1\wedge\overline{\varphi}^3 \\[5pt]
            d\varphi^2 &=&     \varphi^2\wedge\varphi^3-
            					\varphi^2\wedge\overline{\varphi}^3 \\[5pt]
            d\varphi^3 &=&     -i\varphi^1\wedge\overline{\varphi}^1+
            					\varphi^2\wedge\overline{\varphi}^2
       \end{array}\right. \;,
\]
in particular,
\[
 \left\{\begin{array}{rcl}
            \del \varphi^1 &=& i \varphi^1 \wedge \varphi^3 \\[5pt]
            \del \varphi^2 &=& \varphi^2 \wedge \varphi^3 \\[5pt]
            \del \varphi^3 &=& 0
           \end{array}\right.
    \qquad \text{ and } \qquad
    \left\{\begin{array}{rcl}
            \delbar \varphi^1 &=& i\varphi^1\wedge\overline{\varphi}^3\\[5pt]
            \delbar \varphi^2 &=& -\varphi^2\wedge\overline{\varphi}^3\\[5pt]
            \delbar \varphi^3 &=& -i\varphi^1\wedge\overline{\varphi}^1+
            					\varphi^2\wedge\overline{\varphi}^2
           \end{array}\right. \;.
\]
Now fix
the Hermitian metric whose associated fundamental form is
\[
\omega \;:=\; \frac{i}{2} \sum_{j=1}^{3} \varphi^j
\wedge \overline{\varphi}^j \;.
\]
As a matter of notation, from now on we shorten, e.g.,
$\varphi^{1\overline{1}}:=
\varphi^{1}\wedge\overline{\varphi}^1$.\\
As regards the Bott-Chern cohomology, thanks to \cite[Theorem 1.3]{angella-kasuya} we have that the sub-complex
\[
\iota \colon \wedge \left\langle \varphi^1,\, \varphi^2,\, \varphi^3,\, \bar\varphi^1,\, \bar\varphi^2,\, \bar\varphi^3 \right\rangle \hookrightarrow A^{\bullet,\bullet}(X)
\]
is such that $H_{BC}(\iota)$ is an isomorphism, hence,
by explicit computations, we get
%\[
%\begin{array}{lcl}
%H_{BC}^{0,0}(X) & = &  \displaystyle
%\cc \left\langle \left[ 1\right] \right\rangle, \\
%H_{BC}^{1,1}(X) & = &  \displaystyle
%\cc \left\langle \left[
%\varphi^{1\overline{1}}\right],
%\left[
%\varphi^{2\overline{2}}\right] \right\rangle,\\
%H_{BC}^{2,1}(X) & = &  \displaystyle
%\cc \left\langle \left[
%\varphi^{23\overline{2}}+i\varphi^{13\overline{1}}\right] \right\rangle,\\
%H_{BC}^{1,2}(X) & = &  \displaystyle
%\cc \left\langle \left[
%\varphi^{2\overline{2}\overline{3}}-i\varphi^{1\overline{1}
%\overline{3}}\right] \right\rangle,\\
%H_{BC}^{2,2}(X) & = &  \cc \left\langle \left[
%\varphi^{12\overline{1}\overline{2}}\right] \right\rangle,\\
%H_{BC}^{3,2}(X) & = &  \cc \left\langle \left[
%\varphi^{123\overline{1}\overline{2}}\right] \right\rangle,\\
%H_{BC}^{2,3}(X) & = &  \cc \left\langle \left[
%\varphi^{12\overline{1}\overline{2}\overline{3}}\right] \right\rangle,\\
%H_{BC}^{3,3}(X) & = &  \cc \left\langle \left[
%\varphi^{123\overline{1}\overline{2}\overline{3}}\right] \right\rangle.
%\end{array}
%\]
\[
\displaystyle\begin{array}{lcl}
H_{BC}^{0,0}(X) & = &  \displaystyle
\cc \left\langle \left[ 1\right] \right\rangle, \\[5pt]
H_{BC}^{1,1}(X) & = &  \displaystyle
\cc \left\langle \left[
\varphi^{1\overline{1}}\right],
\left[
\varphi^{2\overline{2}}\right] \right\rangle,\\[10pt]
\displaystyle H_{BC}^{2,1}(X) & = &  \displaystyle
\cc \left\langle \left[
\varphi^{23\overline{2}}+i\varphi^{13\overline{1}}\right] \right\rangle,\\[10pt]
\displaystyle H_{BC}^{1,2}(X) & = &  \displaystyle
\cc \left\langle \left[
\varphi^{2\overline{2}\overline{3}}-i\varphi^{1\overline{1}
\overline{3}}\right] \right\rangle,\\[10pt]
H_{BC}^{2,2}(X) & = &  \cc \left\langle \left[
\varphi^{12\overline{1}\overline{2}}\right] \right\rangle,\\[10pt]
H_{BC}^{3,2}(X) & = &  \cc \left\langle \left[
\varphi^{123\overline{1}\overline{2}}\right] \right\rangle,\\[10pt]
H_{BC}^{2,3}(X) & = &  \cc \left\langle \left[
\varphi^{12\overline{1}\overline{2}\overline{3}}\right] \right\rangle,\\[10pt]
H_{BC}^{3,3}(X) & = &  \cc \left\langle \left[
\varphi^{123\overline{1}\overline{2}\overline{3}}\right] \right\rangle.
\end{array}
\]
The other Bott-Chern cohomology groups are trivial.\\
Notice that the Hermitian metric $g$ associated to $\omega$ is
geometrically-$H_{BC}$-formal, hence
$\mathbb{S}^3\times\mathbb{S}^3$ is
geometrically-$H_{BC}$-formal. Now our purpose is to prove
that geometrical-$H_{BC}$-formality is not
stable under small deformations of the complex structure.
In order to get this result, let $J_t$ be the almost complex structure on $X$
defined as
\[
\left\{\begin{array}{rcl}
            \varphi^1_t &:=& \varphi^1 \\[5pt]
            \varphi^2_t &:=& \varphi^2 \\[5pt]
            \varphi^3_t &:=& \varphi^3-t\overline{\varphi}^3
       \end{array}
\right.
\]
then
\[
\left\{\begin{array}{rcl}
            d\varphi^1_t &=& \frac{i(\bar{t}+1)}{1-\mid t\mid^2}
            \varphi ^{13}_t + \frac{i(t+1)}{1-\mid t\mid^2}
            \varphi ^{1\bar{3}}_t \\[10pt]
            d\varphi^2_t &=& \frac{1-\bar{t}}{1-\mid t\mid^2}
            \varphi ^{23}_t + \frac{t-1}{1-\mid t\mid^2}
            \varphi ^{2\bar{3}}_t \\[10pt]
            d\varphi^3_t &=& (t-i)
            \varphi ^{1\bar{1}}_t + (t+1)
            \varphi ^{2\bar{2}}_t
           \end{array}\right. \;.
\]
and consequently $J_t$ is integrable. Set $X_t=(X,J_t)$ and $g_t$
the Hermitian metric whose fundamental form is $\omega_t=
\frac{i}{2}\sum\varphi_t^j\wedge\overline{\varphi}_t^j$. By applying again 
\cite[Theorem 1.3]{angella-kasuya}, we compute the Bott-Chern cohomology of $X_t$;\\
if $\lvert t\rvert^2+\Re\mathfrak{e}\,t-\Im\mathfrak{m}\,t\neq 0$ we get
%\[
%\begin{array}{lcl}
%H_{BC}^{0,0}(X,J_t) & = &  \displaystyle
%\cc \left\langle \left[ 1\right] \right\rangle, \\
%H_{BC}^{1,1}(X,J_t) & = &  \displaystyle
%\cc \left\langle \left[
%\varphi^{1\overline{1}}_t\right],
%\left[
%\varphi^{2\overline{2}}_t\right] \right\rangle,\\
%H_{BC}^{2,1}(X,J_t) & = &  \displaystyle
%\cc \left\langle \left[
%\varphi^{23\overline{2}}_t+i\varphi_t^{13\overline{1}}\right] \right\rangle,\\
%H_{BC}^{1,2}(X,J_t) & = &  \displaystyle
%\cc \left\langle \left[
%\varphi^{2\overline{2}\overline{3}}_t-i\varphi_t^{1\overline{1}
%\overline{3}}\right] \right\rangle,\\
%H_{BC}^{3,2}(X,J_t) & = &  \cc \left\langle \left[
%\varphi_t^{123\overline{1}\overline{2}}\right] \right\rangle,\\
%H_{BC}^{2,3}(X,J_t) & = &  \cc \left\langle \left[
%\varphi_t^{12\overline{1}\overline{2}\overline{3}}\right] \right\rangle,\\
%H_{BC}^{3,3}(X,J_t) & = &  \cc \left\langle \left[
%\varphi_t^{123\overline{1}\overline{2}\overline{3}}\right] \right\rangle,
%\end{array}
%\]
\[
\begin{array}{lcl}
H_{BC}^{0,0}(X_t) & = &  \displaystyle
\cc \left\langle \left[ 1\right] \right\rangle, \\[10pt]
H_{BC}^{1,1}(X_t) & = &  \displaystyle
\cc \left\langle \left[
\varphi^{1\overline{1}}_t\right],
\left[
\varphi^{2\overline{2}}_t\right] \right\rangle,\\[10pt]
H_{BC}^{2,1}(X_t) & = &  \displaystyle
\cc \left\langle \left[
\varphi^{23\overline{2}}_t+\frac{i-t}{t+1}\varphi_t^{13\overline{1}}\right] \right\rangle,\\[10pt]
H_{BC}^{1,2}(X_t) & = &  \displaystyle
\cc \left\langle \left[
\varphi^{2\overline{2}\overline{3}}_t-\frac{i+\bar{t}}{t+1}\varphi_t^{1\overline{1}
\overline{3}}\right] \right\rangle,\\[10pt]
H_{BC}^{3,2}(X_t) & = &  \cc \left\langle \left[
\varphi_t^{123\overline{1}\overline{2}}\right] \right\rangle,\\[10pt]
H_{BC}^{2,3}(X_t) & = &  \cc \left\langle \left[
\varphi_t^{12\overline{1}\overline{2}\overline{3}}\right] \right\rangle,\\[10pt]
H_{BC}^{3,3}(X_t) & = &  \cc \left\langle \left[
\varphi_t^{123\overline{1}\overline{2}\overline{3}}\right] \right\rangle,
\end{array}
\]
where the other groups are trivial, in particular $H^{2,2}_{BC}(X,J_t)$
vanishes.\\
%Notice that, with respect to the metric $g_t$, $X_t$ is not
%$H_{BC}$-geometrically-formal, in fact we can consider
%the wedge product
%\[
%\varphi^{1\overline{1}}_t \wedge \varphi^{2\overline{2}}_t=
%-\varphi_t^{12\overline{1}\overline{2}}
%\]
%which is not harmonic with respect to the Laplacian $\Delta_{BC}^{g_t}$.\\
Now we will show that there are no geometrical-$H_{BC}$-formal
Hermitian metric on $X_t$.
To this purpose
we are going to exhibit a non-trivial ABC-Massey triple product.\\
Setting
$$
 \mathfrak{a}_{12} \;=\; \left[\varphi^{1\bar{1}}_t\right] \in H^{1,1}_{BC}
 (X_t) \;,
$$
$$
 \mathfrak{a}_{23} \;=\; \left[\varphi^{2\bar{2}}_t\right] \in H^{1,1}_{BC}
(X_t)\;,
$$
$$
 \mathfrak{a}_{34} \;=\; \left[\varphi^{2\bar{2}}_t\right] \in H^{1,1}_{BC}
 (X_t) \;. $$\\
we get $ \mathfrak{a}_{12}\cup\mathfrak{a}_{23}=\mathfrak{a}_{23}\cup
\mathfrak{a}_{34}=0$. Indeed
\[
\del\delbar\varphi^{3\bar{3}}_t=\left[(t-i)(\bar{t}+1)+(t+1)(\bar{t}+i)\right]
\varphi^{12\bar{1}\bar{2}}_t=:A_t\varphi^{12\bar{1}\bar{2}}_t,
\]
so, in the hypothesis that $\lvert t\rvert^2+\Re\mathfrak{e}\,t-\Im\mathfrak{m}\,t\neq 0$,
we can take as representatives
\[
\alpha_{13}=-\frac{1}{A_t}\varphi^{3\bar{3}}_t, \qquad
\alpha_{24}=0.
\]
Thus the corresponding ABC-Massey product is
 \begin{eqnarray*}
%\mathfrak{a}_{1234}
 \;
 \left\langle \mathfrak{a}_{12},\, \mathfrak{a}_{23},\, \mathfrak{a}_{34} \right\rangle_{ABC}&=&
 \left[ -\frac{1}{A_t}\varphi^{23\bar{2}\bar{3}}_t\right] 
 \in \frac{H^{2, 2}_{A}}{H^{1,1}_{BC} \smile H^{1, 1}_{A} + H^{1,1}_{A}
 \smile H^{1,1}_{BC}} \left(X_t\right)\;.
 \end{eqnarray*}
It is easy to check that this ABC-Massey triple product is not zero, in fact
$\left[ -\frac{1}{A_t}\varphi^{23\bar{2}\bar{3}}_t\right]\neq 0$ in
$H^{2, 2}_{A}(X_t)$ since $-\frac{1}{A_t}\varphi^{23\bar{2}\bar{3}}_t$
is $\tilde\Delta_{A}^{g_t}$-harmonic in $X_t$.
Furthermore $H_A^{1,1}(X_t)=\left\lbrace 0\right\rbrace$,
since $H_{BC}^{2,2}(X_t)=\left\lbrace 0\right\rbrace$.\\
This concludes the proof.
\end{proof}
As a consequence, we obtain the following
\begin{cor}\label{instability}
The property of geometric-$H_{BC}$-formality is not stable under small deformations
of the complex structure. 
\end{cor} 
\begin{rem}
Recall that a Hermitian metric $g$ on a complex manifold $X$ of dimension $n$ is said to be {\em strong K\"ahler with torsion} (shortly {\em SKT}), respectively 
{\em Gauduchon}, if its fundamental form $\omega$ satisfies $\del\delbar \omega =0$, respectively $\del\delbar \omega^{n-1} =0$.\newline 
A direct computation shows that the Hermitian metric $\omega :=\frac{i}{2} \sum_{j=1}^{3} \varphi^j
\wedge \overline{\varphi}^j$ defined on $\mathbb{S}^3\times\mathbb{S}^3$
is SKT and Gauduchon. As regard the deformation previously
considered, there are two different situations for $X_t$ in a
small neighborhood of $t=0$. If $\lvert t\rvert^2+
\Re\mathfrak{e}\,t-\Im\mathfrak{m}\,t\neq 0$
the Hermitian
metric $\omega_t:=\frac{i}{2}\sum\varphi_t^j\wedge\overline{\varphi}_t^j$ is not SKT, and, more precisely,
$X_t$ does not admit such a metric.

Indeed, let
\[
\del\delbar\varphi_t^{3\bar 3}=-2\left(\lvert t\rvert^2+
\Re\mathfrak{e}\,t-\Im\mathfrak{m}\,t \right)
\varphi_t^{1 \bar 1 2 \bar 2}=:U_t;
\]
then, for $\lvert t\rvert^2+
\Re\mathfrak{e}\,t-\Im\mathfrak{m}\,t > 0$, the $(2,2)$-form $U_t$ gives rise to a $\del\delbar$-exact $(1,1)$-positive non-zero current on $X_t$. Then, in view of 
the
characterisation Theorem of the existence of SKT metrics in terms of currents (see \cite{E} and \cite{Pop}), it follows that $X_t$ has no SKT metrics for $\lvert t\rvert^2+
\Re\mathfrak{e}\,t-\Im\mathfrak{m}\,t > 0$.

Otherwise, if $\lvert t\rvert^2+
\Re\mathfrak{e}\,t-\Im\mathfrak{m}\,t= 0$, a straightforward computation shows that the Hermitian metric $\omega_t$ is SKT.

\end{rem}


\begin{thebibliography}{12}
%\bibitem{angella-tomassini}
%D. Angella, A. Tomassini, On cohomological decomposition of almost-complex manifolds and deformations, {\em J. Symplectic Geom.} {\bfseries 9} (2011), 403--428.
%\bibitem{angella-tomassini-1} D. Angella, A. Tomassini, On the cohomology of almost-complex manifolds,
%{\em Int. J. Math.} \textbf{23} (2012), no.~2.
\bibitem{A} A. Aeppli, On the cohomology structure of Stein manifolds 1965 Proc. Conf. Complex Analysis (Minneapolis, Minn., 1964). 58--70 
\bibitem{angella-kasuya} D. Angella, H. Kasuya Bott-Chern cohomology of solvmanifolds,  {\tt preprint} arXiv:1212.5708.
\bibitem{AT} D. Angella, A. Tomassini, On Bott-Chern cohomology and formality, {\tt preprint}.
\bibitem{BC}
R. Bott, S.~S. Chern, Hermitian vector bundles and the equidistribution of the zeroes of their holomorphic sections, {\em Acta Math.} \textbf{114} (1965), no.~1, 71--112.
\bibitem{C-E} E. Calabi, B. Eckmann, A class of compact, complex manifolds which are not algebraic,
{\em Ann. of Math.} \textbf{58} (1953), 494--500.
\bibitem{DGMS}
P. Deligne, Ph.~A. Griffiths, J. Morgan, D.~P. Sullivan, Real homotopy theory of K\"ahler manifolds, {\em Invent. Math.} \textbf{29} (1975), no.~3, 245--274.
%\bibitem{joyce} D.D. Joyce, {\em Compact Manifolds with Special Holonomy}, Oxford Mathematical Monographs. Oxford University Press, Oxford, 2000. xii+436 pp..
\bibitem{E} N. Egidi  Egidi,  Special metrics on compact complex manifolds, {\em Differential Geom. Appl.} {\bf 14} (2001), no.~3, 217--234.
\bibitem{K}
D. Kotschick, On products of harmonic forms, {\em Duke Math. J.} \textbf{107} (2001), no.~3, 521--531.
\bibitem{Nak} I. Nakamura, Complex parallelisable manifolds and their small deformations, {\em J. Differential Geom.} \textbf{10} (1975), 85--112.
\bibitem{NT} J. Neisendorfer, L. Taylor, Dolbeault homotopy theory, {\em Trans. Amer. Math. Soc.} \textbf{245} (1978), 183--210.
\bibitem{P} M. Parton, Explicit parallelizations on product if spheres and Calabi-Eckmann structures,
{\em Rend. Istit. Mat. Univ. Trieste} \textbf{XXXV} (2003), 61--67.
\bibitem{Pop} D. Popovici, Deformation openness and closedness of various classes of compact complex manifolds; examples, {\em Ann. Sc. Norm. Super. Pisa Cl. Sci. (5)} {\bf 13} (2014), no.~2, 255--305.
\bibitem{S} M. Schweitzer, Autour de la cohomologie de Bott-Chern, {\tt preprint} arXiv:0709.3528v1[math. AG]
\bibitem{TT} A. Tomassini, S. Torelli, On Dolbeault formality and small deformations, {\em Internat. J. Math.} {\bf 25} (2014), 9 p..
\end{thebibliography}
\end{document}